\documentclass[12pt,reqno]{amsart}
\usepackage{amsfonts}
\usepackage{amssymb,latexsym}
\usepackage{enumerate}
\usepackage{mathrsfs}
\usepackage{amsmath}
\usepackage{amsthm}

\usepackage{hyperref}

\hypersetup{hypertex=true,colorlinks=true,linkcolor=blue,anchorcolor=blue,citecolor=blue}

scaled\magstephalf

\newtheorem{theorem}{Theorem}[section]

\newtheorem{lemma}{Lemma}

\newtheorem{proposition}{Proposition}[section]

\numberwithin{equation}{section}

\begin{document}

\setlength{\baselineskip}{1.2\baselineskip}
\title[On the solvability for a $p$-$k$-Hessian inequalityArticle Title]{On the solvability for a $p$-$k$-Hessian inequality}

	\author{Zhenghuan Gao}
	\address{School of Mathematics and Statistics\\
		Xi'an Jiaotong University\\
		Xi'an 710049, CHINA}
	\email{gzhmath@xjtu.edu.cn }
	
	\author{Shujun Shi}
	\address{School of Mathematical Sciences\\
		Harbin Normal University\\
		Harbin 150025, Heilongjiang Province, CHINA}
	\email{shjshi@hrbnu.edu.cn}
	
	\author{Yuzhou Zhang}
	\address{School of Mathematical Sciences\\
		Harbin Normal University\\
		Harbin 150025, Heilongjiang Province, CHINA}
	\email{yuzhouzhang@stu.hrbnu.edu.cn}

	\date{}

	\maketitle
	
	\begin{abstract}	In this paper, we discuss the solvability of a $p$-$k$-Hessian entire inequality. We prove  that the inequality with sub-lower-critical exponent admits no negative solutions. Moreover, the exponent is sharp. The proof is based on choosing suitable test functions and integrating by parts.
\end{abstract}


	\baselineskip 17pt
	
	\section{Introduction}
	In this paper, we consider the following differential inequality
	
	\begin{equation}\label{mainineq}
		F_{k,p}[u]\geq (-u)^\alpha\quad\text{in }\mathbb R^n,
	\end{equation}
	where $F_{k,p}[u]=\big[D(|Du|^{p-2}Du)\big]_k$ is sum of $k$-th principal minors of
	the matrix $D(|Du|^{p-2}Du)$.
	Let $\lambda$ be the eigenvalues of $D(|Du|^{p-2}Du)$. Then $F_{k,p}[u]=\sigma_k(\lambda)$,
	where $\sigma_k(\lambda)$ is the $k$-th elementary symmetric polynomial of $\lambda\in\mathbb R^n$,$$\sigma_k(\lambda)=\sum_{1\leq i_1<\cdots<i_k\leq n}\lambda_{i_1}\cdots\lambda_{i_k}.$$
	We call $F_{k,p}[u]$ the $p$-$k$-Hessian operator.
	It is a generalization of Laplacian operator, $p$-Laplacian operator, Monge-Amp\`ere operator and $k$-Hessian operator. And it is firstly introduced by Trudinger and Wang \cite{TrudingerWang1999Ann}.
	
	When $p=2$, $k=1$, \eqref{mainineq} is the Laplacian inequality $\Delta u\geq (-u)^\alpha$. In the celebrated paper \cite{GidasSpruck1981CPAM}, Gidas and Spruck  stated beautiful and deep results about the inequality. When $k=1$, \eqref{mainineq} coincides with the $p$-Laplacian inequality $\Delta_pu\geq (-u)^\alpha$. There are some splendid results about the $p$-Laplacian inequality by Serrin and Zou \cite{SerrinZou2002Acta}, and we first state one of them.
	
	\begin{theorem}[Serrin-Zou \cite{SerrinZou2002Acta}]
		Assume $p<n$, then the differential inequality $$\Delta_pu\geq (-u)^\alpha\quad\text{in }\mathbb R^n$$ has a negative solution in $\mathbb R^n$ if and only if $\alpha>\frac{n(p-1)}{n-p}$.
	\end{theorem}
	
	When $p=2$, $2$-$k$-Hessian operator is the usual $k$-Hessian operator. Let $\Omega$ be a domain in $\mathbb R^n$.
	A function $u$  is called $k$-admissible in $\Omega$ if $u\in C^2(\Omega)$ and $\lambda(D^2u)\in\Gamma_k$ in $\Omega$, where
	$$\Gamma_k:=\{\lambda\in\mathbb R^n\mid \sigma_l(\lambda)\geq 0,l=1,\cdots,k\}.$$
	Let $\Phi^k(\Omega)$ be the class of $k$-admissible functions
	$$\Phi^k(\Omega):=\Big\{u\in C^2(\Omega)\Big|\lambda\big(D^2u(x)\big)\in \Gamma_k,\forall\ x\in\Omega\Big\}.$$
	When $2k<n$, denote $k^*:=\frac{n(k+1)}{n-2k}$, $k*=\frac{(n+2)k}{n-2k}$ and $k_*:=\frac{nk}{n-2k}$. $k^*$ is the critical exponent for Sobolev embedding inequality established by Wang \cite{Wangxujia1994Indiana}. $k*$ is the critical exponent introduced by Tso \cite{Tso1990AIHP} where the author considered the solvability of the Dirichlet problem $\sigma_k(D^2u)=(-u)^p$ in $\Omega$ and $u|_{\partial\Omega}=0$. $k_*$ is called lower critical exponent.
	For $k$-Hessian inequality $\sigma_k(D^2u)\geq (-u)^\alpha$, Phuc and Verbitsky \cite{PhucVerbitsky2006CPDE,PhucVerbitsky2008Ann} proved that it has no negative solution for $\alpha\in (k,k_*]$ and showed that  $k_*$ is sharp. Ou \cite{Ouqianzhong2010MAA} obtained the nonexistence result with sub-lower critical exponent via different approach.
	\begin{theorem}[Phuc-Verbitsky \cite{PhucVerbitsky2006CPDE,PhucVerbitsky2008Ann}, Ou \cite{Ouqianzhong2010MAA}]
		Assume $2k<n$, then the inequality $$\sigma_k(D^2u)\geq (-u)^\alpha\quad\text{in }\mathbb R^n$$ admits no negative entire solution in $\Phi^k(\mathbb R^n)$ for any $\alpha\in(-\infty,k_*]$. Moreover, $k_*$ is sharp.
	\end{theorem}
	
	The paper is concerned with the $p$-$k$-Hessian inequality \eqref{mainineq}.
	A function $u$ is called $p$-$k$-admissible in $\Omega$ if $u\in C^1(\Omega)$, $|Du|^{p-2}Du\in C^1(\Omega)$ and $\lambda \big(D(|Du|^{p-2}Du)\big)\in\Gamma_k$ in $\Omega$.
	Let $\Phi^{p,k}(\Omega)$ be the class of $p$-$k$-admissible functions. That is
	$$\Phi^{p,k}(\Omega):=\Big\{u\in C^1(\Omega)\Big| |Du|^{p-2}Du\in C^1(\Omega), \lambda\Big(D\big(|Du|^{p-2}Du(x)\big)\Big)\in\Gamma_k,\forall\ x\in\Omega\Big\}.$$
	If $pk<n$, denote by $k_{p,*}:=\frac{n(p-1)k}{n-pk}$.

	Now we state our main theorem as follows:
	\begin{theorem}\label{mainth}
		If $pk<n$, then \eqref{mainineq} has no negative solution in $\Phi^{p,k}(\mathbb R^n)$ if and only if $\alpha\leq k_{p,*}$.
	\end{theorem}
	We are going to prove Theorem \ref{mainth} by using local integral estimates and their asymptotic behavior. To conduce these sharp estimates, we need to establish some iteration forms on the $p$-$k$-Hessian inequality. The approach is widely used in partial differential equation.
	For history of this idea and strategy, one can refer to Obata \cite{Obata1962JMSJ}, Mitidieri and Pohozaev \cite{MitidieriPohozaev2004MJM}. The iteration technique for $k$-Hessian inequality firstly arose in \cite{ChangGurskyYang2003NSAM} by Chang, Gursky and Yang, see also Gonz\'alez \cite{Gonzalez2005Duke}, Ou \cite{Ouqianzhong2010MAA,Ouqianzhong2013PJM}.
	
	Recently, Bao and Feng \cite{BaoFeng2022CMB} gave the necessary and sufficient conditions on global solvability for another type $p$-$k$-Hessian inequalities
	\begin{equation}\label{eq::BaoFeng}
		F_{k,p}[u]\geq f(u)\quad\text{in }\mathbb R^n.
	\end{equation}
	They proved that if $f(t)$ is monotonically non-decreasing in $(0,\infty)$, then \eqref{eq::BaoFeng} admits a positive $p$-$k$-admissible solution if and only if for any $a>0$,
	\begin{equation}\label{cond::KO}
		\int_a^\infty\bigg(\int_0^tf^k(s)\mathrm ds\bigg)^{-\frac1{(p-1)k+1}}\mathrm dt=\infty.
	\end{equation}
	Condition \eqref{cond::KO} is a generalized Keller-Osserman condition. The original  Keller-Osserman condition was proposed by Keller \cite{Keller1957CPAM} and Osserman \cite{Osserman1957PJM} when dealing with $\Delta u\geq f(u)$ in $\mathbb R^n$. Naito and Usami \cite{NaitoUsami1997MathZ} extended it to the case of $p$-Laplacian. Ji and Bao \cite{JiBao2010PAMS} generalized it to the $k$-Hessian case. However, these are different issues from ours.
	
	Now we outline the content of the paper. In Section \ref{sec2}, we state some facts concerning the elementary symmetric functions on non-symmetric matrices and $p$-$k$-Hessian operator. In Section \ref{sec3}, we prove Theorem \ref{mainth}.

	\section{Preliminary}\label{sec2}
	For any $k=1,\cdots, n$, and $\lambda=(\lambda_1,\cdots, \lambda_n)$, the $k$-th elementary symmetric function on $\lambda$ is defined by
	\begin{equation}
		\sigma_k(\lambda)=\sum_{1\leq i_1<i_2<\cdots<i_k\leq n}\lambda_{i_1}\lambda_{i_2}\cdots\lambda_{i_k}.
	\end{equation}
	Let $A=(a_{ij})\in\mathbb R^n$ be an $n\times n$ matrix, we define $\sigma_k(A)$ by the sum of $k$-th principal minors of $A$. That is
	\begin{equation}
		\sigma_k(A)=\frac1{k!}\sum_{1\leq i_1,\cdots, i_{k}, j_1,\cdots, j_k\leq n}\delta_{i_1i_2\cdots i_k}^{j_1j_2\cdots j_k}a_{i_1j_1}a_{i_2j_2}\cdots a_{i_kj_k},
	\end{equation}
	where $\delta_{i_1i_2\cdots i_k}^{j_1j_2\cdots j_k}$ is the Kronecker symbol, which has the value $+1$ (respectively, $-1$) if $i_1,i_2,\cdots,i_k$
	are distinct and $(j_1 j_2\cdots j_k)$ is an even permutation (respectively, an odd permutation) of
	$(i_1i_2\cdots i_k)$, and has the value $0$ in any other cases. We use the convention that $\sigma_0(A)=1$.
	We denote the eigenvalues of $A$ by $\lambda(A)$ and if $\lambda(A)$ are all real, it is clear that $\sigma_k(\lambda(A))=\sigma_k(A)$, see \cite{PietraGavitoneXia2021Adv}.
	Denote by $\sigma_k^{ij}(A)=\frac{\partial\sigma_k(A)}{\partial a_{ij}}$.It is easy to see that
	\begin{equation}\label{formula::sigmak::1}
		\sigma_k(A)=\frac1k\sum_{i,j=1}^n\sigma_k^{ij}(A)a_{ij}.
	\end{equation}
	For general (non-symmeric) matrices, Pietra, Gavitone and Xia proved the following result in \cite{PietraGavitoneXia2021Adv}.
	\begin{proposition}[Pietra, Gavitone and Xia \cite{PietraGavitoneXia2021Adv}]\label{prop::expand}
		For any $n\times n$ matrix $A=(a_{ij})$, we have
		\begin{equation}\label{eq::expand}
			\sigma_k^{ij}(A)=\sigma_{k-1}(A)\delta_{ij}-\sum_{l=1}^n\sigma_{k-1}^{il}(A)a_{jl}.
		\end{equation}
	\end{proposition}
	\begin{proposition}\label{prop::change}
		For any $n\times n$ matrix $A=(a_{ij})$, we have
		\begin{equation}\label{eq::change}
			\sigma_k^{il}(A)a_{jl}=\sigma_k^{lj}(A)a_{li}.
		\end{equation}
	\end{proposition}
	\begin{proof}
		We prove it by induction. For $k=1$, $\sigma_k^{il}=\delta_{il}$, \eqref{eq::change} holds naturally. Suppose it holds for $k-1$, then
		\begin{equation}
			\begin{aligned}
				\sigma_{k}^{il}(A)a_{jl}=&(\sigma_{k-1}(A)\delta_{il}-\sigma_{k-1}^{is}(A)a_{ls})a_{jl}\\
				=&\sigma_{k-1}(A)a_{ji}-\sigma_{k-1}^{is}(A)a_{ls}a_{jl}\\
				=&\sigma_{k-1}(A)a_{ji}-\sigma_{k-1}^{sl}(A)a_{si}a_{jl}.
			\end{aligned}
		\end{equation}
		While
		\begin{equation}
			\begin{aligned}
				\sigma_k^{lj}(A)a_{li}=&(\sigma_{k-1}(A)\delta_{lj}-\sigma_{k-1}^{lm}(A)a_{jm})a_{li}\\
				=&\sigma_{k-1}(A)a_{ji}-\sigma_{k-1}^{lm}(A)a_{jm}a_{li}.
			\end{aligned}
		\end{equation}
		Hence \eqref{eq::change} holds for $k$.
	\end{proof}
	\begin{proposition}\label{eq::divergence-free}
		Suppose $X=(X_1,\cdots, X_n)$ is a vector field, $A=(a_{ij})$ is an $n\times n$ matrix with $a_{ij}={\partial_ j}X_i:=X_{i,j}$. Then
		\begin{equation}\label{eq::divergence-free}
			\partial_j\sigma_k^{ij}(A)=0.
		\end{equation}
	\end{proposition}
	\begin{proof}
		We prove it also by induction. For $k=1$, $\sigma_1^{ij}(A)=\delta_{ij}$, then $\partial_j \sigma_1^{ij}(A)=\partial_j\delta_{ij}=0$. Suppose \eqref{eq::divergence-free} holds for $k-1$, that is
		\begin{equation}\label{induction-assumption-prop-divergence-free}
			\partial_j\sigma_{k-1}^{ij}(A)=0.
		\end{equation}
		Then
		\begin{equation}
			\begin{aligned}
				\partial_j\sigma_k^{ij}(A)=&\partial _j (\sigma_{k-1}\delta_{ij}-\sigma_{k-1}^{il}a_{jl})=\partial_j(\sigma_{k-1}\delta_{ij}-\sigma_{k-1}^{lj}a_{li})\\
				=&\partial_i\sigma_{k-1}(A)-\partial_j\sigma_{k-1}^{lj}(A)a_{li}-\sigma_{k-1}^{lj}\partial_ja_{li}\\=&\partial_i\sigma_{k-1}(A)-\sigma_{k-1}^{lj}\partial_ia_{lj}=0,
			\end{aligned}
		\end{equation}
		where we use proposition \ref{prop::change} in the second equality, we use \eqref{induction-assumption-prop-divergence-free} and the fact $\partial_j a_{li}=\partial_j(\partial_iX_{l})=\partial_i(\partial_jX_l)=\partial_ia_{lj}$ in the fourth equality. Thus \eqref{eq::divergence-free} holds for $k$.
	\end{proof}
	
	\begin{proposition}\label{prop::domination}
		Suppose $u\in\Phi^{p,k}(\Omega)$. Then in $\Omega\setminus\{x:|Du(x)|=0\}$, there holds
		\begin{equation}\label{eq::domination}
			\sigma_s^{ij}(D(|Du|^{p-2}Du))\leq \sum_{l=1}^n\sigma_s^{ll}(D(|Du|^{p-2}Du)),\quad s=1,\cdots,k,\ i,j=1,\cdots,n.
		\end{equation}
	\end{proposition}
	\begin{proof}
		For any point $x_0$ in $\Omega\setminus\{x:|Du(x)|=0\}$, we claim that $u\in C^1$ and $|Du|^{p-2}Du\in C^1$ if and only if $u\in C^2$. The necessity is obvious. Now we prove the sufficiency. For any $i,j\in \{1,\cdots,n\}$, since $(|Du|^{p-2}u_i)_j$ exists at $x_0$, we have
		\begin{equation}
			\begin{aligned}
				&\lim_{t\rightarrow 0} \frac{|Du(x_0+te_j)|^{p-2}u_i(x_0+te_j)-|Du(x_0)|^{p-2}u_i(x_0)}t\\
				=&|Du(x_0)|^{p-2}(\delta_{il}-(p-2)\frac{u_i(x_0)u_l(x_0)}{|Du(x_0)|^2})\lim_{t\rightarrow 0}\frac{u_l(x_0+te_j)-u_l(x_0)}t
			\end{aligned}
		\end{equation}
		Let
		\begin{equation}
			B_{ij}:=|Du|^{p-2}(\delta_{ij}+(p-2)\frac{u_iu_j}{|Du|^2}).
		\end{equation}
		Since  $B_{ij}$ is positive definite provided $p>1$, $B$ is invertible. Thus we prove the sufficiency.
		
		In the following, we do the following computations at $x_0$. Note that
		\begin{equation}
			(|Du|^{p-2}u_i)_j=|Du|^{p-2}u_{ij}+(p-2)|Du|^{p-4}u_mu_{mj}u_i=B_{im}u_{mj}.
		\end{equation}
		If we take $B^\frac12$ and $B^{-\frac12}$  positive definite matrices,
		\begin{equation}\label{eq::eq2}
			B^\frac12_{ij}=|Du|^\frac{p-2}2(\delta_{ij}+(\sqrt{p-1}-1)|Du|^{-2}u_iu_j),
		\end{equation}
		and
		\begin{equation}\label{eq::eq3}
			B^{-\frac12}_{ij}=|Du|^\frac{2-p}2(\delta_{ij}-\frac{\sqrt{p-1}-1}{\sqrt{p-1}}\frac{u_iu_j}{|Du|^2}).
		\end{equation}
		It is easy to check that the matrices defined in \eqref{eq::eq2} and \eqref{eq::eq3} are the  square root matrices of $B$ and $B^{-1}$ respectively.
		Then
		\begin{equation}\label{eq::eq1}
			\frac{\partial(B^\frac12D^2uB^\frac12)_{ml}}{\partial (BD^2u)_{ij}}=B^{-\frac12}_{mi}B^{\frac12}_{jl}.
		\end{equation}
		Since the eigenvalues of $D(|Du|^{p-2}Du)$ and $B^\frac12D^2uB^\frac12$ are the same, we obtain,
		$$F_{k,p}[u]=\sigma_k(D(|Du|^{p-2}Du))=\sigma_k(B^\frac12D^2uB^\frac12).$$
		If we denote by $\sigma_k^{ij}:=\frac{\partial\sigma_k(BD^2u)}{\partial(BD^2u)_{ij}}$ and $\tilde\sigma_k^{ij}:=\frac{\partial\sigma_k(B^\frac12D^2uB^\frac12)}{\partial(B^\frac12D^2uB^\frac12)_{ij}}$. It follows from $\lambda(D(|Du|^{p-2}Du))\in\Gamma_k$ that $\tilde \sigma_k^{ij}$ is positive defined. So
		\begin{equation}\label{eq::eq4}
			|\tilde \sigma_k^{ij}|\leq \sum_{l=1}^n\tilde\sigma_k^{ll}.
		\end{equation}
		By \eqref{eq::eq1}, we get
		\begin{equation}
			\sigma_{k}^{ij}=\tilde\sigma_k^{ml}B^{-\frac12}_{im}B^\frac12_{jl}.
		\end{equation}
		By \eqref{eq::eq2}, \eqref{eq::eq3} and \eqref{eq::eq4}, we obtain
		\begin{equation}
			|\sigma_k^{ij}|\le \sum_{m=1}^n\tilde\sigma_k^{mm}=\sum_{m=1}^n\sigma_k^{mm}.
		\end{equation}
	\end{proof}


	\section{Proof of Theorem \ref{mainth}}\label{sec3}
	
	In this section, we first prove \eqref{mainineq} admits negative admissible solution when  $\alpha>k_{p,*}$.
	
	Let $w=-C_*\big(A+|x|^\frac p{p-1}\big)^{-\frac{k(p-1)}{\alpha-(p-1)k}}$, $C_*$ is to be determined later. Then
		$$\begin{aligned}
			(|Dw|^{p-2}w_i)_j=&\tilde CC_*^{p-1}\big(A+|x|^\frac p{p-1}\big)^{-\frac{\alpha(p-1)}{\alpha-(p-1)k}}\delta_{ij}\\
			&-\frac{\alpha(p-1)}{\alpha-(p-1)k}\frac{p}{p-1}\tilde CC_*^{p-1}\big(A+|x|^\frac p{p-1}\big)^{-\frac{\alpha(p-1)}{\alpha-(p-1)k}-1}|x|^\frac{p}{p-1}\frac{x_ix_j}{|x|^2}\\
			=&\tilde CC_*^{p-1}\big(A+|x|^\frac p{p-1}\big)^{-\frac{\alpha(p-1)}{\alpha-(p-1)k}-1}\\&\quad\bigg(\big(A+|x|^\frac p{p-1}\big)\delta_{ij}-\frac{\alpha(p-1)}{\alpha-(p-1)k}\frac{p}{p-1}{|x|^\frac p{p-1}}\frac{x_ix_j}{|x|^2}\bigg)\\
			=&\tilde CC_*^{p-1}\big(A+|x|^\frac p{p-1}\big)^{-\frac{\alpha(p-1)}{\alpha-(p-1)k}}\\&\quad\bigg(\delta_{ij}-\frac{\alpha(p-1)}{\alpha-(p-1)k}\frac{p}{p-1}\frac{|x|^\frac p{p-1}}{A+|x|^\frac p{p-1}}\frac{x_ix_j}{|x|^2}\bigg),
		\end{aligned}$$
		where $\tilde C=\frac{kp}{\alpha-(p-1)k}\big|\frac{kp}{\alpha-(p-1)k}\big|^{p-2}$. By direct computation,
			$$\begin{aligned}
				F_{k,p}[w]=&\tilde C^kC_*^{k(p-1)}C_n^k\big(A+|x|^\frac p{p-1}\big)^{-\frac{\alpha(p-1)(k-1)}{\alpha-(p-1)k}-(k-1)}\bigg(\big(A+|x|^\frac p{p-1}\big)-\frac kn\frac{\alpha(p-1)}{\alpha-(p-1)k}\frac{p}{p-1}{|x|^\frac p{p-1}}\bigg)\\
				=&\tilde C^kC_*^{k(p-1)}C_n^k\big(A+|x|^\frac p{p-1}\big)^{-\frac{\alpha(p-1)(k-1)}{\alpha-(p-1)k}-(k-1)}|x|^\frac p{p-1}\bigg(1-\frac kn\frac{\alpha(p-1)}{\alpha-(p-1)k}\frac{p}{p-1}\bigg)\\
				&+A\tilde C^kC_*^{k(p-1)}C_n^k\big(A+|x|^\frac p{p-1}\big)^{-\frac{\alpha(p-1)(k-1)}{\alpha-(p-1)k}-(k-1)}.
			\end{aligned}	$$
			When $\alpha>k_{p,*}$, $w\in \Phi^{p,k}(\mathbb R^n)$ and $F_{k,p}[w]\geq (-w)^\alpha$ for $A\in [0,\infty)$ after we choose a suitable $C_*=\big(\frac{(n-1)!}{k!(n-k)!}\frac{(pk)^{(p-1)k}(\alpha(n-pk)-n(p-1)k)}{(\alpha-(p-1)k)^{(p-1)k+1}}\big)^\frac1{\alpha-(p-1)k}>0$.
			
			Now we prove there is no negative admissible solution to \eqref{mainineq}   when  $\alpha\leq k_{p,*}$.
			Assume $u<0$ be a solution of \eqref{mainineq} in $\Gamma_k$. In the following, we use $\sigma_k$ and $\sigma_k^{ij}$ instead of $\sigma_k(D(|Du|^{p-2}Du))$ and $\frac{\partial \sigma_k(D(|Du|^{p-2}Du))}{\partial (|Du|^{p-2}u_i)_j}$ respectively for convenience.
			
			Let $\eta$ be a $C^2$ cut-off function satisfying:
			\begin{equation}\label{eq::prop::eta}
				\begin{cases}
					\eta\equiv 1&\quad\text{in }B_R,\\
					0\leq \eta\leq 1&\quad\text{in }B_{2R},\\
					\eta\equiv 0&\quad\text{in }\mathbb R^n\backslash B_{2R},\\
					|D\eta|\leq \frac{C}{R}&\quad\text{in }\mathbb R^n,
				\end{cases}
			\end{equation}
			where and in the following $C$ is a constant independent of $R$ and $u$.
			
			For $s=1,\cdots, k$, denote
			\begin{equation}\label{eq::expression::Es}
				\begin{split}
					&B_s:=\int_{\mathbb R^n}\sigma_{k-s}|Du|^{sp}(-u)^{-\delta-s}\eta^\theta,\\
					&M_s:=\int_{\mathbb R^n}\sigma_{k-s+1}^{ij}|Du|^{sp-2}u_iu_j(-u)^{-\delta-s}\eta^\theta,\\
					&E_s:=\int_{\mathbb R^n}\sigma_{k-s+1}^{ij}|Du|^{sp-2}u_i\eta_j(-u)^{-\delta-s+1}\eta^{\theta-1},\\
					&b_s:=\frac1{s!}(1-\frac1p)^{s-1}(k-\frac{k-s}{p})\prod_{j=0}^{s-1}(\delta+j),
				\end{split}
			\end{equation}
			and
			\begin{equation}
				e_1:=\theta,\, e_s:=\theta\frac1{(s-1)!}(1-\frac1p)^{s-1}\prod_{j=0}^{s-2}(\delta+j)\quad s=2,\cdots,k.
			\end{equation}
			
			We have the following expansion,
			\begin{lemma}
				\begin{equation}\label{eq::expansion}
					k\int_{\mathbb R^n}\sigma_k(-u)^{-\delta}\eta^\theta=-\sum_{s=1}^kb_sB_s-\sum_{s=1}^ke_sE_s.
				\end{equation}
			\end{lemma}
			\begin{proof}
				Let
				\begin{equation}\label{eq::ep::g}
					g=(-u)^{-\delta-s}\eta^\theta.
				\end{equation}
				By proposition \ref{prop::expand}, we have
				\begin{equation}\label{eq::ori::Ms}
					\begin{aligned}
						M_s=&\int_{\mathbb{R}^n}\sigma_{k-s+1}^{ij}|Du|^{sp-2}u_iu_jg\\
						=&\int_{\mathbb{R}^n}\big(\sigma_{k-s}\delta_{ij}-\sigma_{k-s}^{il}(|Du|^{p-2}u_j)_l\big)|Du|^{sp-2}u_iu_jg\\
						=&\int_{\mathbb{R}^n}\sigma_{k-s}|Du|^{sp}g-\int_{\mathbb{R}^n}\sigma_{k-s}^{il}(|Du|^{p-2}u_j)_l|Du|^{sp-2}u_iu_jg.
					\end{aligned}
				\end{equation}
				Note that
				\begin{equation}\label{eq::com::term-in-Ms}
					\begin{aligned}
						&\sigma_{k-s}^{il}(|Du|^{p-2}u_j)_l|Du|^{sp-2}u_iu_jg=\frac{p-1}{sp}\sigma_{k-s}^{il}|Du|^{p-2}u_i(|Du|^{sp})_lg\\
						=&\frac{p-1}{sp}(\sigma_{k-s}^{il}|Du|^{(s+1)p-2}u_ig)_l-\frac{p-1}{sp}\sigma_{k-s}^{il}(|Du|^{p-2}u_i)_l|Du|^{sp}g-\frac{p-1}{sp}\sigma_{k-s}^{il}|Du|^{(s+1)p-2}u_ig_l
					\end{aligned}
				\end{equation}
				Putting \eqref{eq::ep::g} and \eqref{eq::com::term-in-Ms} into \eqref{eq::ori::Ms}, we get
				\begin{equation}
					\begin{aligned}
						M_s=&\frac1s(k-\frac{k-s}p)\int_{\mathbb{R}^n}\sigma_{k-s}|Du|^{sp}(-u)^{-\delta-s}\eta^\theta\\
						&+(\delta+s)\frac1s(1-\frac1p)\int_{\mathbb{R}^n}\sigma_{k-s}^{ij}|Du|^{(s+1)p-2}u_iu_j(-u)^{-\delta-s-1}\eta^\theta\\
						&+\theta\frac1s(1-\frac1p)\int_{\mathbb{R}^n}\sigma_{k-s}^{ij}|Du|^{(s+1)p-2}u_i\eta_j(-u)^{-\delta-s}\eta^{\theta-1}.
					\end{aligned}
				\end{equation}
				That is
				\begin{equation}\label{eq::expansion::Ms}
					M_s=\frac1s(k-\frac{k-s}p)B_s+(\delta+s)\frac1s(1-\frac1p)M_{s+1}+\theta\frac1s(1-\frac1p)E_{s+1}.
				\end{equation}
				On the other hand, by \eqref{formula::sigmak::1}
				we obtain,
				\begin{equation}\label{eq::expansion::1}
					\begin{aligned}
						k\int_{\mathbb{R}^n}\sigma_k(-u)^{-\delta}\eta^\theta=&\int_{\mathbb{R}^n}\sigma_k^{ij}(|Du|^{p-2}u_i)_j(-u)^{-\delta}\eta^{\theta}\\
						=&-\delta\int_{\mathbb{R}^n}\sigma_k^{ij}|Du|^{p-2}(-u)^{-\delta-1}u_iu_j\eta^\theta-\theta\int_{\mathbb{R}^n}\sigma_{k}^{ij}|Du|^{p-2}u_i\eta_j(-u)^{-\delta}\eta^{\theta-1}\\
						=&-\delta M_1-\theta E_1.
					\end{aligned}
				\end{equation}
				Substituting \eqref{eq::expansion::Ms} into \eqref{eq::expansion::1} iteratively, we get \eqref{eq::expansion}. This finish the proof.
			\end{proof}
			Now we prove Theorem \ref{mainth}.
			\begin{proof}
				Multiply both sides of \eqref{mainineq} by $k(-u)^{-\delta}\eta^\theta$, we have
				\begin{equation}\label{eq::mainineq-1}
					k\int_{\mathbb{R}^n}(-u)^{\alpha-\delta}\eta^\theta\leq k\int_{\mathbb{R}^n}\sigma_k(-u)^{-\delta}\eta^\theta.
				\end{equation}
				Now we estimate the error term $E_s$. By proposition \ref{prop::domination} and $|D\eta|\leq \frac{C}{R}$, we have
				\begin{equation}
					|E_s|\leq \frac{C}{R}\int_{\mathbb{R}^n}\sigma_{k-s}|Du|^{sp-1}(-u)^{-\delta-s+1}\eta^{\theta-1}.
				\end{equation}
				Using Young's inequality with exponent pair $(\frac{ps}{ps-1},ps)$, $\forall \varepsilon>0$, the last inequaliy turns into
				\begin{equation}\label{eq::est::Es-1}
					|E_s|\leq \varepsilon\int_{\mathbb{R}^n}\sigma_{k-s}|Du|^{sp}(-u)^{-\delta-s}\eta^\theta+\frac{C_\varepsilon}{R^{sp}}\int_{\mathbb{R}^n}\sigma_{k-s}(-u)^{-\delta-(1-p)s}\eta^{\theta-sp}.
				\end{equation}
				Now we deal with the last term in \eqref{eq::est::Es-1}. Note that
				\begin{equation}
					\begin{aligned}
						&\frac{1}{R^{sp}}\int_{\mathbb{R}^n}\sigma_{k-s}(-u)^{-\delta-(1-p)s}\eta^{\theta-sp}\\\cong& \frac{1}{R^{sp}}\int_{\mathbb{R}^n}\sigma_{k-s}^{ij}(|Du|^{p-2}u_i)_j(-u)^{-\delta-(1-p)s}\eta^{\theta-sp}\\
						=&-\frac{\delta+s-sp}{R^{sp}}\int_{\mathbb{R}^n}\sigma_{k-s}^{ij}|Du|^{p-2}u_iu_j(-u)^{-\delta-s-1+sp}\eta^{\theta-sp}\\&-\frac{\theta-sp}{R^{sp}}\int_{\mathbb{R}^n}\sigma_{k-s}^{ij}|Du|^{p-2}u_i\eta_j(-u)^{-\delta-s+sp}\eta^{\theta-sp-1}\\
						\lesssim& \frac{1}{R^{sp}}\int_{\mathbb{R}^n}\sigma_{k-s-1}|Du|^p(-u)^{-\delta-s-1+sp}\eta^{\theta-sp}+\frac{1}{R^{sp+1}}\int_{\mathbb{R}^n}\sigma_{k-s-1}|Du|^{p-1}(-u)^{-\delta-s+sp}\eta^{\theta-sp-1}\\
						\leq &\varepsilon \int_{\mathbb{R}^n}\sigma_{k-s-1}|Du|^{p(s+1)}(-u)^{-\delta-s-1}\eta^{\theta}+\frac{C_\varepsilon}{R^{(s+1)p}}\int_{\mathbb{R}^n}\sigma_{k-s-1}(-u)^{-\delta+(p-1)(s+1)}\eta^{\theta-(s+1)p},
					\end{aligned}
				\end{equation}
				where we apply Young's inequality with exponent pairs $(s+1,\frac{s+1}{s})$ and $(\frac{p}{p-1}(s+1), \frac{p(s+1)}{ps+1})$ respectively in the last inequality, and we use "$\lesssim$" , "$\cong$", etc. to drop out some positive constants independent of $R$ and $u$. Repeating the same process, we get
				\begin{equation}\label{eq::est::last-term-in-est::Es-1}
					\begin{aligned}
						&\frac{1}{R^{sp}}\int_{\mathbb{R}^n}\sigma_{k-s}(-u)^{-\delta-(1-p)s}\eta^{\theta-sp}\\
						\leq &\varepsilon \int_{\mathbb{R}^n}\sigma_{k-s-1}|Du|^{p(s+1)}(-u)^{-\delta-s-1}\eta^{\theta}+\frac{C_\varepsilon}{R^{(s+1)p}}\int_{\mathbb{R}^n}\sigma_{k-s-1}(-u)^{-\delta+(p-1)(s+1)}\eta^{\theta-(s+1)p}\\
						\leq &\varepsilon \int_{\mathbb{R}^n}\sigma_{k-s-1}|Du|^{p(s+1)}(-u)^{-\delta-s-1}\eta^{\theta}+\varepsilon \int_{\mathbb{R}^n}\sigma_{k-s-2}|Du|^{p(s+2)}(-u)^{-\delta-s-2}\eta^{\theta}\\&+\frac{C_\varepsilon}{R^{(s+2)p}}\int_{\mathbb{R}^n}\sigma_{k-s-2}(-u)^{-\delta+(p-1)(s+2)}\eta^{\theta-(s+2)p}\\
						\leq& \varepsilon\sum_{i=s+1}^kB_i+\frac{C_\varepsilon}{R^{pk}}\int_{\mathbb{R}^n}(-u)^{-\delta+(p-1)k}\eta^{\theta-kp}.
					\end{aligned}
				\end{equation}
				Putting  \eqref{eq::expansion}, \eqref{eq::est::Es-1} and \eqref{eq::est::last-term-in-est::Es-1} together, we obtain
				\begin{equation}\label{eq::mainineq-2}
					k\int_{\mathbb{R}^n}\sigma_k(-u)^{-\delta}\eta^\theta+\sum_{s=1}^k(b_s-\varepsilon)B_s\lesssim \frac{1}{R^{kp}}\int_{\mathbb{R}^n}(-u)^{-\delta+(p-1)k}\eta^{\theta-kp}.
				\end{equation}
				
				Now for $\alpha\in (-\infty,k_{p,*}]$ we divide the proof into four cases.
				\begin{itemize}
					\item[(a)] $\alpha=(p-1)k$,
					\item[(b)] $\alpha\in (-\infty,(p-1)k)$,
					\item[(c)] $\alpha\in((p-1)k,k_{p,*})$,
					\item[(d)] $\alpha=k_{p,*}$.
				\end{itemize}
				
				In all cases of (a)-(c), $b_s>0$ for all $s=1,\cdots,k$. For \textbf{cases (a)}, let $\delta=\alpha$, $\theta>n$. It follows from \eqref{eq::mainineq-1} and \eqref{eq::mainineq-2} that
				\begin{equation}
					\int_{\mathbb{R}^n}\eta^\theta\leq \int_{\mathbb{R}^n}\sigma_k(-u)^{-\delta}\eta^{\theta}\lesssim \frac{1}{R^{pk}}\int_{\mathbb{R}^n}\eta^{\theta-kp}\lesssim \varepsilon\int_{\mathbb{R}^n}\eta^\theta+R^{n-\theta},
				\end{equation}
				where we use Young's inequality with exponent pair $(\frac{\theta}{kp},\frac\theta{\theta-kp})$ in the last inequality. Now we choose $\varepsilon$ small enough and let $R\rightarrow \infty$, we get a contradiction.
				
				For \textbf{cases (b) and (c)}, let $\delta >\frac{n-kp}{kp}(k_{p,*}-\alpha)$ and $0<\delta<\frac{n-kp}{kp}(k_{p,*}-\alpha)$ respectively, then we have $\frac{\alpha-\delta}{(p-1)k-\delta}>1$. By Young's inequality with $(\frac{\alpha-\delta}{(p-1)k-\delta}, \frac{\alpha-\delta}{\alpha-(p-1)k})$, we obtain
				\begin{equation}\label{eq::est::last-term-in-mainineq-2}
					\begin{aligned}
						\frac1{R^{pk}}\int_{\mathbb{R}^n}(-u)^{-\delta+(p-1)k}\eta^{\theta-kp}\leq \varepsilon\int_{\mathbb{R}^n}(-u)^{\alpha-\delta}\eta^\theta+\frac{C_\varepsilon}{R^\frac{kp(\alpha-\delta)}{\alpha-(p-1)k}}\int_{\mathbb{R}^n}\eta^{\theta-\frac{kp(\alpha-\delta)}{\alpha-(p-1)k}}.
					\end{aligned}
				\end{equation}
				Put \eqref{eq::mainineq-1}, \eqref{eq::est::last-term-in-mainineq-2} into \eqref{eq::mainineq-2}, and let $R\rightarrow \infty$, we get a contradiction after choosing $\theta>\frac{kp(\alpha-\delta)}{\alpha-(p-1)k}$.
				
				For \textbf{cases (d)}, first choose  $\delta =0$. It follows from \eqref{eq::expansion::1} that
				\begin{equation}\label{eq::id::E1}
					k\int_{\mathbb{R}^n} \sigma_k\eta^\theta=-\theta E_1.
				\end{equation}
				Next fix $\delta \in (0,\min\{p-1,\frac{k^2p(p-1)}{(n-pk)(kp-1)}\})$, and let $\beta=\frac{n\delta}{\alpha}$. By \eqref{eq::domination}, \eqref{eq::prop::eta}, \eqref{eq::expression::Es} and Young's inequality with exponent $(\frac{p}{p-1},p)$, we have
				\begin{equation}\label{eq::est::E1-1}
					\begin{aligned}
						R^\beta |E_1|\leq &R^{\beta-1}\int_{\mathbb{R}^n}\sigma_{k-1}|Du|^{p-1}\eta^{\theta-1}\\
						\leq &\int_{\mathbb{R}^n}\sigma_{k-1}|Du|^p(-u)^{-\delta-1}\eta^\theta+R^{(\beta-1)p}\int_{\mathbb{R}^n}\sigma_{k-1}(-u)^{(p-1)(\delta+1)}\eta^{\theta-p}\\
						=&B_1+V_1,
					\end{aligned}
				\end{equation}
				where $V_1:=R^{(\beta-1)p}\int_{\mathbb{R}^n}\sigma_{k-1}(-u)^{(p-1)(\delta+1)}\eta^{\theta-p}$. Let
				\begin{equation}\label{eq::expression::Vs}
					\begin{aligned}
						&V_s:=R^{sp(\beta-1)}\int_{\mathbb{R}^n}\sigma_{k-s}(-u)^{-\delta-s+sp(\delta+1)}\eta^{\theta-ps},
					\end{aligned}
				\end{equation} and
				\begin{equation}\label{eq::expression::Ws}
					\begin{aligned}
						&W_s:=R^{\beta-sp}\int_{\mathbb{R}^n}\sigma_{k-s}(-u)^{(p-1)s}\eta^{\theta-sp}.
					\end{aligned}
				\end{equation}
				We can prove the following two inequalities:
				\begin{equation}\label{eq::est::Vs}
					V_s\lesssim B_{s+1}+V_{s+1}+W_{s+1},
				\end{equation}
				and
				\begin{equation}\label{eq::est::Ws}
					W_s\lesssim B_{s+1}+V_{s+1}+W_{s+1}.
				\end{equation}
				
				By \eqref{formula::sigmak::1}, divergence theorem, \eqref{eq::prop::eta} and Proposition \ref{prop::domination}, we obtain
				\begin{equation}\label{eq::est::Vs-1}
					\begin{aligned}
						V_s=&R^{sp(\beta-1)}\int_{\mathbb{R}^n}\sigma_{k-s}(-u)^{-\delta-s+sp(\delta+1)}\eta^{\theta-ps}\\
						\cong &R^{sp(\beta-1)}\int_{\mathbb{R}^n}\sigma_{k-s}^{ij}(|Du|^{p-2}u_i)_j(-u)^{-\delta-s+sp(\delta+1)}\eta^{\theta-ps}\\
						\cong &R^{sp(\beta-1)}\int_{\mathbb{R}^n}\sigma_{k-s}^{ij}|Du|^{p-2}u_iu_j(-u)^{-\delta-s-1+sp(\delta+1)}\eta^{\theta-ps}\\
						&+R^{sp(\beta-1)}\int_{\mathbb{R}^n}\sigma_{k-s}^{ij}|Du|^{p-2}u_i\eta_j(-u)^{-\delta-s+sp(\delta+1)}\eta^{\theta-ps-1}\\
						\lesssim &R^{sp(\beta-1)}\int_{\mathbb{R}^n}\sigma_{k-s-1}|Du|^{p}(-u)^{-\delta-s-1+sp(\delta+1)}\eta^{\theta-ps}\\
						&+R^{sp(\beta-1)-1}\int_{\mathbb{R}^n}\sigma_{k-s-1}|Du|^{p-1}(-u)^{-\delta-s+sp(\delta+1)}\eta^{\theta-ps-1}.
					\end{aligned}
				\end{equation}
				Applying Young's inequality with exponent pair $(s+1, \frac{s+1}s)$ to the first term of last line in \eqref{eq::est::Vs-1}, we derive that
				\begin{equation}\label{eq::est::first-term-in-est::Vs}
					\begin{aligned}
						&R^{sp(\beta-1)}\int_{\mathbb{R}^n}\sigma_{k-s-1}|Du|^{p}(-u)^{-\delta-s-1+sp(\delta+1)}\eta^{\theta-ps}\\
						\lesssim& \int_{\mathbb{R}^n}\sigma_{k-s-1}(-u)^{-\delta-s-1}|Du|^{(s+1)p}\eta^\theta+R^{(s+1)p(\beta-1)}\int_{\mathbb{R}^n}\sigma_{k-s-1}(-u)^{-\delta-s-1+(s+1)p(\delta+1)}\eta^{\theta-(s+1)p}\\
						=&B_{s+1}+V_{s+1}.
					\end{aligned}
				\end{equation}
				Similarly, applying Young's inequality with exponent pair $(\frac{(s+1)p}{p-1}, \frac{(s+1)p}{sp+1})$ to the last term of last line in \eqref{eq::est::Vs-1}, we derive that
				\begin{equation}\label{eq::est::last-term-in-est::Vs}
					\begin{aligned}
						&R^{sp(\beta-1)-1}\int_{\mathbb{R}^n}\sigma_{k-s-1}|Du|^{p-1}(-u)^{-\delta-s+sp(\delta+1)}\eta^{\theta-ps-1}\\
						\lesssim&\int_{\mathbb{R}^n}\sigma_{k-s-1}(-u)^{-\delta-s-1}|Du|^{(s+1)p}\eta^\theta+R^{\frac{\beta s(s+1)p^2}{sp+1}-(s+1)p}\int_{\mathbb{R}^n}\sigma_{k-s-1}(-u)^{(p-1)(s+1)+(\frac{s(s+1)p^2}{sp+1}-1)\delta}\eta^{\theta-(s+1)p}\\
						\lesssim& \int_{\mathbb{R}^n}\sigma_{k-s-1}(-u)^{-\delta-s-1}|Du|^{(s+1)p}\eta^\theta+R^{(s+1)p(\beta-1)}\int_{\mathbb{R}^n}\sigma_{k-s-1}(-u)^{-\delta-s-1+(s+1)p(\delta+1)}\eta^{\theta-(s+1)p}\\&+R^{\beta-(s+1)p}\int_{\mathbb{R}^n}\sigma_{k-s-1}(-u)^{(s+1)(p-1)}\eta^{\theta-(s+1)p}\\=&B_{s+1}+V_{s+1}+W_{s+1},
					\end{aligned}
				\end{equation}
				where the last inequality follows from using Young's inequality with exponent  $(\frac{s(s+1)p^2+p-1}{s(s+1)p^2-(sp+1)},\\\frac{s(s+1)p^2+p-1}{(s+1)p})$.
				Substituting \eqref{eq::est::first-term-in-est::Vs} and \eqref{eq::est::last-term-in-est::Vs} into \eqref{eq::est::Vs-1}, we arrive at \eqref{eq::est::Vs}.
				
				By \eqref{formula::sigmak::1}, divergence theorem, \eqref{eq::prop::eta} and Proposition \ref{prop::domination}, we obtain
				\begin{equation}\label{eq::est::Ws-1}
					\begin{aligned}
						W_s=&R^{\beta-sp}\int_{\mathbb{R}^n}\sigma_{k-s}(-u)^{(p-1)s}\eta^{\theta-sp}\\
						\cong &R^{\beta-sp}\int_{\mathbb{R}^n}\sigma_{k-s}^{ij}(|Du|^{p-2}u_i)_j(-u)^{(p-1)s}\eta^{\theta-sp}\\
						\cong &R^{\beta-sp}\int_{\mathbb{R}^n}\sigma_{k-s}^{ij}|Du|^{p-2}u_iu_j(-u)^{(p-1)s-1}\eta^{\theta-sp}+R^{\beta-sp}\int_{\mathbb{R}^n}\sigma_{k-s}^{ij}|Du|^{p-2}u_i\eta_j(-u)^{(p-1)s}\eta^{\theta-sp-1}\\
						\lesssim& R^{\beta-sp}\int_{\mathbb{R}^n}\sigma_{k-s-1}|Du|^{p}(-u)^{(p-1)s-1}\eta^{\theta-sp}+R^{\beta-sp-1}\int_{\mathbb{R}^n}\sigma_{k-s-1}|Du|^{p-1}(-u)^{(p-1)s}\eta^{\theta-sp-1}.
					\end{aligned}
				\end{equation}
				Applying Young's inequality with exponent pair $(s+1, \frac{s+1}s)$ to the first term of last line in \eqref{eq::est::Ws-1}, we derive that
				\begin{equation}\label{eq::est::first-term-in-est::Ws}
					\begin{aligned}
						&R^{\beta-sp}\int_{\mathbb{R}^n}\sigma_{k-s-1}|Du|^{p}(-u)^{(p-1)s-1}\eta^{\theta-sp}\\
						\lesssim& \int_{\mathbb{R}^n}\sigma_{k-s-1}|Du|^{(s+1)p}(-u)^{-\delta-s-1}\eta^\theta+R^{\frac{(s+1)\beta}s-(s+1)p}\int_{\mathbb{R}^n}\sigma_{k-s-1}(-u)^{\frac{\delta}s+(s+1)(p-1)}\eta^{\theta-(s+1)p}\\
						\lesssim& \int_{\mathbb{R}^n}\sigma_{k-s-1}|Du|^{(s+1)p}(-u)^{-\delta-s-1}\eta^\theta+R^{\beta-(s+1)p}\int_{\mathbb{R}^n}\sigma_{k-s-1}(-u)^{(s+1)(p-1)}\eta^{\theta-(s+1)p}\\&+R^{(s+1)p(\beta-1)}\int_{\mathbb{R}^n}\sigma_{k-s-1}(-u)^{-\delta-s-1+(s+1)p(\delta+1)}\eta^{\theta-(s+1)p}\\
						=&B_{s+1}+W_{s+1}+V_{s+1},
					\end{aligned}
				\end{equation}
				where the last inequality follows from using Young's inequality with exponent pair $(s(s+1)p-s,\frac {s(s+1)p-s}{(sp-1)(s+1)})$ for $R^{\frac{\beta}{s}}(-u)^{\frac{\delta}{s}}\cdot1$.  Then applying Young's inequality with exponent pair $(\frac{(s+1)p}{p-1},\frac{(s+1)p}{sp+1})$ to $|Du|^{p-1}\cdot R^{\beta-(sp+1)}(-u)^{sp+1+\delta}\eta^{-(sp+1)}$, we obtain
				\begin{equation}\label{eq::est::last-term-in-est::Ws}
					\begin{aligned}
						&R^{\beta-sp-1}\int_{\mathbb{R}^n}\sigma_{k-s-1}|Du|^{p-1}(-u)^{(p-1)s}\eta^{\theta-sp-1}\\
						\lesssim&\int_{\mathbb{R}^n}\sigma_{k-s-1}|Du|^{(s+1)p}(-u)^{-\delta-s-1}\eta^\theta\\&+R^{\frac{(s+1)p\beta}{sp+1}-(s+1)p}\int_{\mathbb{R}^n}\sigma_{k-s-1}(-u)^{(\frac{(s+1)p}{sp+1}-1)\delta+(s+1)(p-1)}\eta^{\theta-(s+1)p}\\
						\lesssim& \int_{\mathbb{R}^n}\sigma_{k-s-1}|Du|^{(s+1)p}(-u)^{-\delta-s-1}\eta^\theta+R^{\beta-(s+1)p}\int_{\mathbb{R}^n}\sigma_{k-s-1}(-u)^{(s+1)(p-1)}\eta^{\theta-(s+1)p}\\&+R^{(s+1)p(\beta-1)}\int_{\mathbb{R}^n}\sigma_{k-s-1}(-u)^{-\delta-s-1+(s+1)p(\delta+1)}\eta^{\theta-(s+1)p}\\
						=&B_{s+1}+W_{s+1}+V_{s+1},
					\end{aligned}
				\end{equation}
				where the last inequality follows from using Young's inequality with exponent pair $(\frac {(sp+p-1)(sp+1)}{sp^2(s+1)},\\\frac {(sp+p-1)(sp+1)}{p-1})$.
				Substituting \eqref{eq::est::first-term-in-est::Ws} and \eqref{eq::est::last-term-in-est::Ws} into \eqref{eq::est::Ws-1}, we arrive at \eqref{eq::est::Ws}.
				
				Using \eqref{eq::est::Vs} and \eqref{eq::est::Ws}, we obtain
				\begin{equation}\label{eq::est::Vs-2}
					V_s\lesssim \sum_{i=s+1}^kB_i+V_k+W_k.
				\end{equation}
				Putting \eqref{eq::est::Vs-2} into \eqref{eq::est::E1-1}, we arrive at
				\begin{equation}\label{eq::est::E1-2}
					R^\beta|E_1|\lesssim \sum_{i=1}^kB_i+W_k+V_k.
				\end{equation}
				Note that for fixed $\delta \in (0,\min\{p-1,\frac{k^2p(p-1)}{(n-pk)(kp-1)}\})$, $b_s>0$ and  \eqref{eq::mainineq-2} still holds. So we get
				\begin{equation}\label{eq::est::sumbi}
					\sum_{i=1}^kB_i\lesssim \frac1{R^{kp}}\int_{\mathbb{R}^n}(-u)^{-\delta+(p-1)k}\eta^{\theta-kp}.
				\end{equation}
				Substituting \eqref{eq::expression::Vs}, \eqref{eq::expression::Ws} and \eqref{eq::est::sumbi} into \eqref{eq::est::E1-2}, we have
				\begin{equation}\label{eq::est::E1-3}
					\begin{aligned}
						|E_1|\lesssim& R^{-\beta-kp}\int_{\mathbb{R}^n}(-u)^{-\delta+(p-1)k}\eta^{\theta-kp}+R^{-kp}\int_{\mathbb{R}^n}(-u)^{k(p-1)}\eta^{\theta-kp}\\&+R^{(kp-1)\beta-kp}\int_{\mathbb{R}^n}(-u)^{-\delta-k+kp(\delta+1)}\eta^{\theta-kp}.
					\end{aligned}
				\end{equation}
				Since $\alpha=k_{p,*}=\frac{n(p-1)k}{n-kp}>(p-1)k$, we have $a_1:=\frac{\alpha}{(p-1)k}>1$. Let $b_1$ satisfiy $\frac1{a_1}+\frac1{b_1}=1$, $\theta>b_1kp=\frac{(\alpha-(p-1)k)pk}{\alpha}$, and $\lambda_1=\frac{\theta}{b_1}-kp$. After using H\"older inequality, we derive that
				\begin{equation}\label{eq::est::first-term-in-est::E1-3}
					\begin{aligned}
						&R^{-kp}\int_{\mathbb{R}^n}(-u)^{(p-1)k}\eta^{\theta-pk}\\\leq& R^{-kp}\bigg(\int_{\mathbb{R}^n}\big((-u)^{(p-1)k}\eta^{\theta-kp-\lambda_1}\big)^{a_1}\bigg)^\frac1{a_1}\bigg(\int_{\mathbb{R}^n}\eta^{\lambda_1 {b_1}}\bigg)^\frac1{b_1}
						\\\lesssim&\bigg(\int_{\mathbb{R}^n}(-u)^\alpha\eta^\theta\bigg)^\frac{(p-1)k}\alpha.
					\end{aligned}
				\end{equation}
				Similarly, we have $a_2:=\frac\alpha{(p-1)k-\delta}>1$. Let $b_2$ satisfy $\frac1{a_2}+\frac1{b_2}=1$, $\theta>b_2kp=\frac{kp\alpha}{\alpha+\delta-(p-1)k}$, $\lambda_2=\frac{\theta}{b_2}-kp$. Using H\"older inequality again, we have
				\begin{equation}\label{eq::est::second-term-in-est::E1-3}
					\begin{aligned}
						&R^{-\beta-kp}\int_{\mathbb{R}^n}(-u)^{-\delta+(p-1)k}\eta^{\theta-kp}\\\leq& R^{-\beta-kp}\bigg(\int_{\mathbb{R}^n}\big((-u)^{-\delta+(p-1)k}\eta^{\theta-kp-\lambda_2}\big)^{a_2}\bigg)^\frac1{a_2}\bigg(\int_{\mathbb{R}^n}\eta^{\lambda_2b_2}\bigg)^\frac1{b_2}\\\lesssim&\bigg(\int_{\mathbb{R}^n}(-u)^\alpha\eta^\theta\bigg)^\frac{(p-1)k-\delta}\alpha.
					\end{aligned}
				\end{equation}
				Since $\delta<\frac{k^2p(p-1)}{(n-pk)(kp-1)}$, we have that $a_3:=\frac{\alpha}{(kp-1)\delta+k(p-1)}$. Let $b_3$ satisfy $\frac1{a_3}+\frac1{b_3}=1$, $\theta>b_3kp=\frac{kp\alpha}{\alpha-(kp-1)\delta-k(p-1)}$, $\lambda_3=\frac{\theta}{b_3}$. By  H\"older inequality, we get
				\begin{equation}\label{eq::est::last-term-in-est::E1-3}
					\begin{aligned}
						&R^{(kp-1)\beta-kp}\int_{\mathbb{R}^n}(-u)^{-\delta-k+kp(\delta+1)}\eta^{\theta-kp}\\\leq& R^{(kp-1)\beta-kp}\bigg(\int_{\mathbb{R}^n}\big((-u)^{-\delta-k+kp(\delta+1)}\eta^{\theta-kp-\lambda_3}\big)^{a_3}\bigg)^\frac1{a_3}\bigg(\int_{\mathbb{R}^n}\eta^{\lambda_3b_3}\bigg)^\frac1{b_3}\\\lesssim&\bigg(\int_{\mathbb{R}^n}(-u)^\alpha\eta^\theta\bigg)^\frac{(kp-1)\delta+k(p-1)}\alpha.
					\end{aligned}
				\end{equation}
				Substituting \eqref{eq::est::first-term-in-est::E1-3}, \eqref{eq::est::second-term-in-est::E1-3}, \eqref{eq::est::last-term-in-est::E1-3} into \eqref{eq::est::E1-3}, we obtain
				\begin{equation}\label{eq::est::E1-4}
					|E_1|\lesssim \bigg(\int_{\mathbb{R}^n}(-u)^\alpha\eta^\theta\bigg)^\frac{(p-1)k}\alpha+\bigg(\int_{\mathbb{R}^n}(-u)^\alpha\eta^\theta\bigg)^\frac{(p-1)k-\delta}\alpha+\bigg(\int_{\mathbb{R}^n}(-u)^\alpha\eta^\theta\bigg)^\frac{(kp-1)\delta+k(p-1)}\alpha.
				\end{equation}
				Recall the definition of $E_s$, all the integrations in \eqref{eq::est::E1-4} are taken overline the domain $\mathrm{Supp}(D\eta)=B_{2R}\backslash B_R$. It follows from \eqref{eq::id::E1} and \eqref{eq::est::E1-4} that
				\begin{equation}\label{eq::mainineq::3}
					k\int_{\mathbb R^n}\sigma_k\eta^\theta\lesssim \bigg(\int_{B_{2R}\backslash B_R}(-u)^\alpha\eta^\theta\bigg)^\frac{(p-1)k}\alpha+\bigg(\int_{B_{2R}\backslash B_R}(-u)^\alpha\eta^\theta\bigg)^\frac{(p-1)k-\delta}\alpha+\bigg(\int_{ B_{2R}\backslash B_R}(-u)^\alpha\eta^\theta\bigg)^\frac{(kp-1)\delta+k(p-1)}\alpha.
				\end{equation}
				Since $0<\frac{(p-1)k}\alpha,\frac{(p-1)k-\delta}\alpha,\frac{(kp-1)\delta+k(p-1)}\alpha<1$, from \eqref{eq::mainineq::3}, we have
				\begin{equation}
					\int_{\mathbb R^n}(-u)^\alpha\eta^\theta\leq constant<\infty.
				\end{equation}
				This implies
				\begin{equation}
					\int_{B_{2R}\backslash B_R}(-u)^\alpha\eta^\theta\rightarrow 0\quad\text{as}\quad R\rightarrow +\infty.
				\end{equation}
				Hence, from \eqref{eq::mainineq::3} again, as $ R \to +\infty$, we have
				\begin{equation}
					\int_{\mathbb R^n}(-u)^\alpha\eta^\theta\leq 0.
				\end{equation}
				This contradicts with $u$ being a negative solution.
			\end{proof}
			
			\noindent\textbf{Acknowledgements: } The authors would like to thank Prof. Xi-Nan Ma for his advice, constant support and
				encouragement. The first author  was supported by the National Natural Science Foundation of China under Grant 12301257. The second and third authors were  supported by the National Natural Science Foundation of China under Grant 11971137 and Postgraduate Innovation Grant of Harbin Normal University (HSDBSCX2024-13).

			\bibliographystyle{plain}			
			\bibliography{GaoShiZhang20250302}

			\end{document}